\documentclass[11pt]{article}      

\RequirePackage[colorlinks,citecolor=blue,urlcolor=blue]{hyperref}
\usepackage{graphicx}
\usepackage{enumerate}
\usepackage{amssymb, amsmath, amsthm,  amsfonts}
\usepackage[square,,sort,comma,numbers]{natbib}

\usepackage{subfig}

\newtheorem{proposition}{Proposition}[section] 
\newtheorem{defn}[proposition]{Definition}

\newtheorem{corollary}[proposition]{Corollary}
\newtheorem{remark}[proposition]{Remark}
\newtheorem{example}[proposition]{Example}

\usepackage[T1]{fontenc}
\usepackage{lmodern}

\newcommand{\C}{\ensuremath{{\mathbb C}}}
\newcommand{\R}{\ensuremath{{\mathbb R}}}

\newcommand{\E}{\ensuremath{{\mathbb E}}}
\newcommand{\Pro}{\ensuremath{{\mathbb P}}}

\begin{document}
\title{Probabilistic approach to Appell polynomials}
\author{Bao Quoc Ta\thanks{email: tbao@abo.fi} \\
  \small{ \it \AA bo Akademi University}\\
  \small{\it Department of Natural Sciences/Mathematics and Statistics}\\
  \small{\it FIN-20500 \AA bo, Finland}
  }

\date{}
\maketitle
\begin{abstract}
In this paper we study Appell polynomials by connecting them to random variables. This probabilistic approach yields, e.g., the mean value property which is fundamental in the sense that many other properties can be derived from it. We also discuss moment representations of Appell polynomials. In the latter part of the paper the presented general theory is applied to study some classical explicitly given polynomials.
\\
\textbf{Key words:} {Appell polynomials, gamma distribution,  Bernoulli, Euler, Hermite, Laguerre polynomials.}\\
\textbf{2010 mathematics subject classification:} {11B68, 33C45, 60E05}
\end{abstract}

\section{Introduction}
Paul Appell introduced and studied in his paper \cite{App} published 1880, a system of polynomials $\{Q_n, n=0,1,\dots\}$,  where $Q_n$ is of order $n$, satisfying the recursive differential equation
\[\frac{d}{dx}Q_n(x)=nQ_{n-1}(x), \quad n\geq 1.\]
Such polynomials, now called Appell polynomials, have since Appell's pioneering work appeared and applied in a variety of mathematical fields, e.g., number theory,  numerical analysis and probability theory. Some much studied classical polynomials like Bernoulli, Euler, Hermite polynomials are, in fact, Appell polynomials.  Our main interest in Appell polynomials arises from fairly recent observation due to Novikov and Shiryaev \citep{NS}, see also Kyprianou and Surya \citep{KS} and Salminen \cite{Sa}, that Appell polynomials are of key importance when solving optimal stopping problems with power reward functions and a L\'evy process as the underlying. Another application of Appell polynomials is in constructing time-space martingales for L\'evy processes (see, e.g.,\cite{ Sch, SU}). Picard and Lef\`evre \cite{P-L-97} apply Appell polynomials for the problem of ruin in insurance models. We refer to  \cite{I-V-04,p-L-11, P-L-03} for more applications in insurance mathematics.
 
With the applications above in mind we find it motivated to present in detail in this survey paper basic properties of Appell polynomials. We focus on the probabilistic approach in which to every random variable with finite moments or some exponential moments is associated a family of Appell polynomials,  see Definition 2.1 and 2.3 in the next section. Connecting Appell polynomials to random variables provides us, in particular, when the random variable has some exponential moments, with powerful tools to deeper understanding of the behaviour and properties of Appell polynomials. Our general discussion is completed with examples of classical Appell polynomials: Bernoulli, Euler, Hermite and Laguerre polynomials. It is demonstrated that the probabilistic approach offers us new proofs of some old results for these polynomials but yields also interesting new developments, e.g., the moment representations.

The paper is organized as follows. In the next section Appell polynomials are defined, and many general properties are scrutinized. In particular, we give a simple proofs of the mean value property, find conditions for the moment representation of Appell polynomials and discuss the property  of the unique positive root (which is important in optimal stopping). In section 3 the results in section 2 are illustrated and further analysed for  Bernoulli, Euler, Hermite and Laguerre polynomials.

\section{On Appell polynomials}
\subsection{Definition}
Consider a random variable $\xi$ having some exponential moments, i.e., for some $\lambda>0,$ $M_{\xi}(u):=\E(e^{u\xi})$ $<\infty$ for $|u|<\lambda$. Then, as is well known, $\E(\xi^n)<\infty$ for all $n=1, 2\dots,$ and
\begin{equation}\label{moment-generating}
\E(e^{u\xi})=\sum_{n=0}^{\infty}\frac{u^n}{n!}\E(\xi^n),\quad |u|<\lambda.
\end{equation} 
Taking on the RHS of (\ref{moment-generating}) $u$ to be a complex number then the RHS defines a complex analytic function $z\mapsto \Psi(z), \quad |z|<\lambda, z\in \C$. Since $M_{\xi}(0)=1$ it follows from the continuity of $\Psi$ that $|\Psi(z)|>0$ for $|z|<\lambda$. Consequently, $z\mapsto\frac{1}{\Psi(z)}$ is a well defined analytic function and representable via power series
\[\frac{1}{\Psi(z)}=\sum_{n=0}^{\infty}c_n z^n,\quad |z|<\rho,\]
where $\rho>0$ is the radius of convergence. In particular, for $z=u\in \R$ such that $|u|<\rho$ we obtain 
\begin{equation*}
\frac{1}{\E(e^{u\xi})}=\sum_{n=0}^{\infty}c_n u^n=\sum_{n=0}^{\infty}\hat{c}_n\frac{u^n}{n!},
\end{equation*}
where $\hat{c}_n=c_n n!$. For $x\in \R$ and $|u|<\rho$ it holds
\begin{align}\label{App-0}
\nonumber\frac{e^{ux}}{\E(e^{u\xi})}&=\sum_{l=0}^{\infty}\frac{u^l}{l!}x^l\sum_{k=0}^{\infty}\frac{u^k}{k!}\hat{c}_k\\
\nonumber &=\sum_{l=0}^{\infty}\sum_{k=0}^{\infty}\frac{u^{l+k}}{l!k!}\hat{c}_k x^l\\
&=\sum_{n=0}^{\infty}\frac{u^n}{n!}\sum_{k=0}^{n}\binom{n}{k}\hat{c}_{k}x^{n-k},
\end{align}
where we have used the fact that the value of the sum does not depend on the order of summation since the both series are absolutely convergent in a neighbourhood of the origin. 

Guided by (\ref{App-0}), we give now the definition of the Appell polynomials associated with a random variable.  

\begin{defn}\label{app-def}
Let $\xi$ be a random variable having some exponential moments. The polynomials $Q_{n}^{(\xi)}, n=0, 1,2,\dots,$ satisfying
\begin{equation}\label{appell}
\sum_{n=0}^{\infty}\frac{u^n}{n!}Q_{n}^{(\xi)}(x)=\frac{e^{ux}}{\E(e^{u\xi})},
\end{equation}
where $Q_{n}^{(\xi)}$ is of order $n$, are called the Appell polynomials associated with~ $\xi$.
\end{defn}
\begin{remark}
Recall from, e.g., \citep{Sch} that given two analytic functions $f$ and $g$ such that $g(0)=0, g'(0)\neq 0, f(0)\neq 0$ then the polynomials $S_n, n=1,2,\dots$ satisfying
\begin{equation*}\label{sheffer}
\sum_{n=0}^{\infty}\frac{u^n}{n!}S_n(x)=f(u)e^{xg(u)}
\end{equation*}
are called the Sheffer polynomials. The discussion above leading to Definition \ref{app-def} reveals that Appell polynomials are, in fact, special Sheffer polynomials.
\end{remark}
In the next definition we introduce the Appell polynomials in case $\xi$ has some moments (but not necessarily exponential moments).
\begin{defn}\label{app-def2}
Let $\xi$ be a random variable which has the moments up to $N$, i.e., 
\[\E(|\xi^n|)<\infty, \quad n=0,1,\dots, N.\]
The Appell polynomials $\{Q_{n}^{(\xi)}, n=0, 1,2,\dots, N\}$ associated with $\xi$ are defined via
\begin{align}
\label{app-def2-1}
&{Q}_{0}^{(\xi)}(x)=1\quad \text{for all}\quad x,\\
\label{app-def2-2}
&{Q}_{n}^{(\xi)}(x):=\sum_{i=0}^{n}\binom{n}{i}Q_{n-i}^{(\xi)}(0)x^{i},\quad n=1,\dots, N
\end{align}
where $Q_{j}^{(\xi)}(0), j=1,2,\dots,n$ are generated by the recurrence formula
\begin{equation}\label{app-def2-3}
Q_{j}^{(\xi)}(0)=-\sum_{i=0}^{j-1} \binom{j}{i}Q_{j-i}^{(\xi)}(0)\E(\xi^{i}).
\end{equation}
\end{defn}
It is proved below (see Proposition \ref{equi-definitions}) that in case $\xi$ has the moment generating function, i.e., some exponential moments, then the polynomials in Definition \ref{app-def} and Definition \ref{app-def2} are the same.
\begin{remark}\label{def-kyprianou}
The Appell polynomials associated with a random variable $\xi$ can be defined also via the differential equation (\ref{recursive}) together with the normalisation (see, e.g., Kyprianou \cite[p 250]{kyprianou06})
  $$\E(Q_{n}^{(\xi)}(\xi))=0.$$

\end{remark}

\subsection{General properties}
We now state and prove several general properties of the Appell polynomials, many of them are elementary and can be found in \cite{kyprianou06, NS,Sa} but we wish to give a fairly complete discussion to make the paper self-contained.
\subsubsection{Basic formulas}
 Putting $u=0$ in (\ref{appell}) yields  $Q_{0}^{(\xi)}(x)\equiv 1.$ Letting $\xi\equiv 0$ then clearly
 $$Q_{n}^{(0)}(x)=x^n, \quad n=0,1,2,\dots.$$
 \begin{proposition}  Series expansion:
\begin{equation}\label{series-exp}
Q_{n}^{(\xi)}(x+y)=\sum_{k=0}^{n}\binom{n}{k}Q_{k}^{(\xi)}(x)y^{n-k}.
\end{equation}
Letting $y=(m-1)x$, $m=1,2,\dots$ yields the multiplication formula
\[Q_{n}^{(\xi)}(mx)=\sum_{k=0}^{n}\binom{n}{k}Q_{k}^{(\xi)}(x)(m-1)^{n-k}x^{n-k}.\]
Putting $x=0$ in (\ref{series-exp}) gives
\begin{equation}\label{appell-rep}
Q_{n}^{(\xi)}(y)=\sum_{k=0}^{n}\binom{n}{k}Q_{k}^{(\xi)}(0)y^{n-k}.
\end{equation}
\end{proposition}
\begin{proof}
We have
\begin{align*}
\sum_{n=0}^{\infty}\frac{u^n}{n!}Q_{n}^{(\xi)}(x+y)=\frac{e^{u(x+y)}}{\E(e^{u\xi})}&=\frac{e^{ux}}{\E(e^{u\xi})}e^{uy}\\
&=\sum_{k=0}^{\infty}\frac{u^k}{k!}Q_{k}^{(\xi)}(x)\sum_{l=0}^{\infty}\frac{u^l}{l!}y^l\\
&=\sum_{k=0}^{\infty}\sum_{l=0}^{\infty}\frac{u^{k+l}}{k!l!}Q_{k}^{(\xi)}(x)y^l\\
&=\sum_{n=0}^{\infty}\frac{u^n}{n!}\sum_{l=0}^{n}\binom{n}{l}Q_{n-l}^{(\xi)}(0)y^{l},
\end{align*}
where we have substituted $n=k+l$. Notice that the two last equalities follow since the order of summation can be changed due to the absolute convergence of the series.
\end{proof}
\begin{proposition}
Recursive differential equation:
\begin{equation}\label{recursive}
\frac{d}{dx}Q^{(\xi)}_{n}(x)=nQ^{(\xi)}_{n-1}(x),
\end{equation}
or, equivalently,
\begin{equation*}
Q_{n}^{(\xi)}(x)=Q_{n}^{(\xi)}(0)+n\int_{0}^{x}Q_{n-1}^{(\xi)}(z)dz.
\end{equation*}
\end{proposition}
\begin{proof}
This follows directly from (\ref{appell-rep}) by differentiating. 
\end{proof}

\begin{proposition}
 Mean value property:\\
\begin{equation}\label{meanvalue}
\E(Q^{(\xi)}_{n}(\xi+x))=x^n, \quad n=1,2\dots.
\end{equation}
\end{proposition}

\begin{proof}
{\it Case 1}: $\xi$ has some exponential moments. Putting $x=0$ in (\ref{appell}) yields
\begin{align*}
1=\E(e^{u\xi})\sum_{n=0}^{\infty}\frac{u^n}{n!}Q_{n}^{(\xi)}(0)
&=\sum_{m=0}^{\infty}\frac{u^m}{m!}\E(\xi^m)\sum_{n=0}^{\infty}\frac{u^n}{n!}Q_{n}^{(\xi)}(0).
\end{align*}
Since the series are absolute convergence we have for all $|u|<\lambda$
\[1=\sum_{k=0}^{\infty}u^k \sum_{n=0}^{k}\frac{1}{n!(k-n)!}Q_{n}^{(\xi)}(0)\E(\xi^{k-n}),\]
and, hence,
\begin{equation}\label{sum-mean}
\sum_{n=0}^{k}\frac{1}{n!(k-n)!}Q_{n}^{(\xi)}(0)\E(\xi^{k-n})=0, \quad\forall k\geq 1.
\end{equation}
From (\ref{appell-rep}) we have
\begin{align}\label{pre-mean}
\nonumber \E(Q_{n}^{(\xi)}(x+\xi))&=\sum_{i=0}^{n}\binom{n}{i}Q_{n-i}^{(\xi)}(0)\E((x+\xi)^i)\\
\
\nonumber &=\sum_{i=0}^{n}\binom{n}{i}Q_{n-i}^{(\xi)}(0)\sum_{j=0}^{i}\binom{i}{j}x^j\,\E(\xi^{i-j})\\
\
&=\sum_{j=0}^{n}x^j\sum_{i=j}^{n}\binom{n}{i}\binom{i}{j}Q_{n-i}^{(\xi)}(0)\E(\xi^{i-j}).
\end{align}
Consider $$A:=\sum_{i=j}^{n}\binom{n}{i}\binom{i}{j}Q_{n-i}^{(\xi)}(0)\E(\xi^{i-j}), \quad j<n.$$
Substituting $m=i-j$ yields for $k\geq 1$
\begin{align*}
A&=\sum_{m=0}^{n-j}\binom{n}{m+j}\binom{m+j}{j}Q_{n-j-m}^{(\xi)}(0)\E(\xi^{m})\\
\
&=\frac{n!}{j!}\sum_{m=0}^{n-j}\frac{1}{(n-j-m)!m!} Q_{n-j-m}^{(\xi)}(0)\E(\xi^{m})\\
\
&=\frac{n!}{(n-k)!}\sum_{m=0}^{k}\frac{1}{(k-m)!m!} Q_{k-m}^{(\xi)}(0)\E(\xi^{m})\\
&=0,
\end{align*}
by (\ref{sum-mean}). Consequently, from (\ref{pre-mean}), we have now $\E(Q_{n}^{(\xi)}(x+\xi))) =x^n$.\\
{\it Case 2}: $\xi$ has the moments up to $N$, $N\geq 1$. In this case the Appell polynomials $Q_{n}^{(\xi)}, n=1\dots, N$ are defined by Definition \ref{app-def2}. It is seen that the mean value property holds for $n=1$. Assume that the mean value property holds for $n=m\leq N$. From (\ref{app-def2-2}) and the fact $\E(|\xi|^m)<\infty$, it follows that 
\begin{align*}
\frac{d}{dx}\E(Q_{m+1}^{(\xi)}(x+\xi))&=\frac{d}{dx}\left(\sum_{i=0}^{m+1}\binom{m+1}{i}Q_{m+1-i}^{(\xi)}(0)\sum_{j=0}^{i}\binom{i}{j}\E(\xi^{i-j})x^j\right)\\
&=\sum_{i=1}^{m+1}\binom{m+1}{i}Q_{m+1-i}^{(\xi)}(0)\sum_{j=1}^{i}\binom{i}{j}\E(\xi^{i-j})jx^{j-1}\\
&=\sum_{i=1}^{m+1}\binom{m+1}{i}Q_{m+1-i}^{(\xi)}(0)\ i \sum_{j=1}^{i}\frac{(i-1)!}{(i-j)!(j-1)!}\E(\xi^{i-j})x^{j-1}\\
&=\sum_{i=1}^{m+1}\binom{m+1}{i}Q_{m+1-i}^{(\xi)}(0)\ i\ \E(x+\xi)^{i-1}\\
&=(m+1)\sum_{i=1}^{m+1}\frac{m!}{(m+1-i)!(i-1)!}Q_{m+1-i}^{(\xi)}(0)\E(x+\xi)^{i-1}\\
&=(m+1)\E\big(Q_{m}^{(\xi)}(x+\xi)\big)=(m+1)x^{m}.
\end{align*}
Hence, \[\E(Q_{m+1}^{(\xi)}(x+\xi))=c+x^{m+1},\]
where $c$ is some constant. Since from (\ref{app-def2-2}) and (\ref{app-def2-3})
 \[\E(Q_{m+1}^{(\xi)}(\xi))=\sum_{j=0}^{m+1}\binom{m+1}{j}Q_{m+1-j}^{(\xi)}(0)\E(\xi^j)=0\]
it is seen that $c=0$ and we obtain
$$\E(Q_{m+1}^{(\xi)}(x+\xi))=x^{m+1},$$
as claimed. 
\end{proof}

\begin{remark}
\item {(i)} Using the definition of the Appell polynomials as given in Remark \ref{def-kyprianou}, the mean value property is proved in \cite{kyprianou06} by exploiting the dominated convergence theorem.
\item {(ii)} Notice that by the mean value property and Definition \ref{app-def} it holds
\[\sum_{n=0}^{\infty}\frac{u^n}{n!}\E(Q_{n}^{(\xi)}(x+\xi))=\sum_{n=0}^{\infty}\frac{u^n}{n!}x^n=\E\Big(\frac{e^{u(x+\xi)}}{\E(e^{u\xi})}\Big)=\E\sum_{n=0}^{\infty}\frac{u^n}{n!}Q_{n}^{(\xi)}(x+\xi).\]
\end{remark}

\begin{proposition}\label{equi-definitions}
Assume that $\xi$ has some exponential moments. Denote by $\tilde{Q}_{n}^{(\xi)}$ the polynomial in Definition \ref{app-def} and ${Q}_{n}^{(\xi)}$ the polynomial in Definition \ref{app-def2}. Then for all $x\in \R$  
\begin{equation}\label{equiv-defs}
\tilde{Q}_{n}^{(\xi)}(x)= {Q}_{n}^{(\xi)}(x), n=0,1, 2\dots.
\end{equation}

\end{proposition}
\begin{proof}

Obviously  equality (\ref{equiv-defs}) holds for $n=0,1$. Assume (\ref{equiv-defs}) holds for $n=k$, we will show that (\ref{equiv-defs}) also holds for $n=k+1$. We have 
\[\frac{d}{dx}{Q}_{k+1}^{(\xi)}(x)=(k+1)Q_{k}^{(\xi)}(x),\]
and from (\ref{recursive})
\[\frac{d}{dx}\tilde{Q}_{k+1}^{(\xi)}(x)=(k+1)\tilde{Q}_{k}^{(\xi)}(x).\]
So we obtain
\[\frac{d}{dx}(\tilde{Q}_{k+1}^{(\xi)}(x)-{Q}_{k+1}^{(\xi)}(x))=0,\]	
and, hence,
\[\tilde{Q}_{k+1}^{(\xi)}(x)-{Q}_{k+1}^{(\xi)}(x)=c,\quad \text{ for all}\quad x.\]
Consequently, since both $\tilde{Q}_{k+1}^{(\xi)}(x)$ and ${Q}_{k+1}^{(\xi)}(x)$ satisfy the mean value property, 
\[\E(\tilde{Q}_{k+1}^{(\xi)}(x+\xi))-\E({Q}_{k+1}^{(\xi)}(x+\xi))=c,\]
which yields $c=0$.
\end{proof}

The next result is also an application of the mean value property. 
\begin{proposition}
Representation of powers:
\begin{equation}\label{inverse-formula}
x^n=\sum_{k=0}^{n}\binom{n}{k}Q_{k}^{(\xi)}(x)\E(\xi^{n-k}),
\end{equation}
in other words
\begin{equation}\label{inverse-2}
Q_{n}^{(\xi)}(x)=x^n-\sum_{k=0}^{n-1}\binom{n}{k}Q_{k}^{(\xi)}(x)\E(\xi^{n-k}).\
\end{equation}

\end{proposition}
\begin{proof}
From (\ref{series-exp}) we have
\[Q_{n}^{(\xi)}(x+\xi)=\sum_{k=0}^{n}\binom{n}{k}Q_{k}^{(\xi)}(x)\xi^{n-k}.\]
Taking expectation and using the mean value property yields (\ref{inverse-formula}).
\end{proof}

\begin{example}
Let $\xi$ be a standard log-normally distributed random variable , i.e., $\xi=e^{\eta}$, where $\eta\sim N(0,1)$. Then $\mu_k=\E(\xi^k)=e^{k^2/2}$. Hence, the infinite series
\[\sum_{n=0}^{\infty}\frac{u^n}{n!}\E(\xi^n)=\sum_{n=0}^{\infty}\frac{u^n}{n!}e^{n^2/2}\]
diverges and $\xi$ has no exponential moments. The Appell polynomials $Q_{n}^{(\xi)}, n=1,2\dots$ can  be defined as in Definition \ref{app-def2}. Using the recurrence formula (\ref{inverse-2}) we can calculate
\[Q_{1}^{(\xi)}(x)=x-e^{1/2};\quad  Q_{2}^{(\xi)}(x)=x^2-2x+2e^{1/2}-e^2,\]
\[Q_{3}^{(\xi)}(x)=x^3-3e^{1/2}x^2-3(e^2-2e^{1/2})x-6e-e^{9/2},\dots\]
\end{example}
From (\ref{appell-rep}) it is seen that the coefficients of $Q_{n}^{(\xi)}$ are determined by $Q_{k}^{(\xi)}(0)$, $k=1,2,\dots,n$. To calculate these we may use the triangular system (\ref{app-def2-3}) to obtain for $k=1,2,\dots$
\begin{proposition} It holds
\begin{equation}\label{q-n-zero}
Q_{k}^{(\xi)}(0)=-det\left( \begin{array}{ccccc}
\mu_k & \binom{k}{k-1} \mu_1 &\ldots&\binom{k}{2}\mu_{k-2}  &\binom{k}{1}\mu_{k-1} \\
\mu_{k-1} & 1 &\ldots& \binom{k-1}{2}\mu_{k-3} & \binom{k-1}{1}\mu_{k-2}\\
\vdots&\vdots&\ddots&\vdots&\vdots\\
\mu_2 & 0 &\ldots& 1&\binom{2}{1}\mu_1\\
\mu_1 & 0 &\ldots& 0& 1

 \end{array} \right),\\
 \end{equation}
 where $\mu_k:=\E(\xi^k), k=1,\dots, n$ are the moments of $\xi$.
 \end{proposition}
 \begin{proof}
 From (\ref{app-def2-3}) solving the equation with variables $Q_{k}^{(\xi)}(0)$, $k=1,2,\dots, n$
 \[\sum_{j=1}^{k}\binom{k}{j}Q_{j}^{(\xi)}(0)\E(\xi^{k-j})=-\mu_k,\]
 yields (\ref{q-n-zero}).
 \end{proof}
  We have also the following relation (see \citep[formula 2.9]{Sa}) between the cumulants of $\xi$. Recall that the cummulants $\kappa_n$ of $\xi$ are defined by 
 \begin{equation}\label{cumulants}
 \sum_{n=0}^{\infty}\frac{u^n}{n!}\kappa_n=\log(\E(e^{u\xi})).
 \end{equation}
\begin{proposition} For $n=1,2\dots$ it holds
\begin{equation}\label{cumulant-mu-q}
\kappa_{n+1}=\sum_{j=0}^{n}\binom{n}{j}\mu_{j+1}Q_{n-j}^{(\xi)}(0).
\end{equation}
\end{proposition}
\begin{proof}
Taking derivative in $u$ in (\ref{cumulants}) we obtain
 \[\sum_{n=0}^{\infty}\frac{u^n}{n!}\kappa_{n+1}=\frac{\E(\xi e^{u\xi})}{\E(e^{u\xi})}=\sum_{n=0}^{\infty}\frac{u^n}{n!}\mu_{n+1}\sum_{n=0}^{\infty}\frac{u^n}{n!}Q_{n}^{(\xi)}(0),\]
which yields (\ref{cumulant-mu-q}).
\end{proof}
\subsubsection{Probabilistic properties}
In this section we focus on properties of Appell polynomials induced by the transforms of the underlying random variable. 
 \begin{proposition}\label{symmetric}
Let $\eta:=-\xi$ then for all $n$ and $x$
\begin{equation}\label{relation}
Q^{(\eta)}_{n}(x)=(-1)^nQ^{(\xi)}_{n}(-x).
\end{equation}
In particular,
\[Q^{(\eta)}_{n}(0)=(-1)^nQ^{(\xi)}_{n}(0).\]
\end{proposition}
\begin{proof}
We have
\[\frac{e^{ux}}{\E(e^{u\eta})}=\sum_{n=0}^{\infty}\frac{u^n}{n!}Q_{n}^{(\eta)}(x),\]
and
\[\frac{e^{ux}}{\E(e^{u\eta})}=\frac{e^{-u(-x)}}{\E(e^{-u\xi})}=\sum_{n=0}^{\infty}\frac{u^n}{n!}(-1)^nQ_{n}^{(\xi)}(-x).\]
Hence, (\ref{relation}) holds.
\end{proof}
\begin{corollary}
If $\xi$ is a symmetric random variable then
\[Q^{(\xi)}_{n}(x)=(-1)^nQ^{(\xi)}_{n}(-x),\]
i.e., $Q^{(\xi)}_{2n+1}$ is an odd polynomial and $Q^{(\xi)}_{2n}$ is even, $n=1,2,\dots$ In particular, $Q^{(\xi)}_{2n+1}(0)=0$.
\end{corollary}
\begin{proposition}\label{propo-i.i.d}
Let $\xi_1$ and $\xi_2$ be two independent random variables and put $\xi:=\xi_1+\xi_2$
then
\begin{equation}{\label{independent}}
Q_{n}^{(\xi)}(x+y)=\sum_{k=0}^{n}\binom{n}{k}Q_{k}^{(\xi_1)}(x)Q_{n-k}^{(\xi_2)}(y).
\end{equation}
In particular,
\begin{equation}\label{independent-2}
\sum_{k=0}^{n}\binom{n}{k}Q_{k}^{(\xi_1)}(x)Q_{n-k}^{(\xi_2)}(-x)=\sum_{k=0}^{n}\binom{n}{k}Q_{k}^{(\xi_1)}(0)Q_{n-k}^{(\xi_2)}(0).
\end{equation}
\end{proposition}
\begin{proof} (see also \citep{Sa})
By the independence of $\xi_1$ and $\xi_2$, we have 
\[\frac{e^{u(x+y)}}{\E(e^{u\xi})}=\frac{e^{ux}}{\E(e^{u\xi_1})}\frac{e^{uy}}{\E(e^{u\xi_2})}.\]
 Hence
 \begin{align*}
\sum_{n=0}^{\infty}\frac{u^n}{n!}Q_{n}^{(\xi)}(x+y)&=\sum_{k=0}^{\infty}\frac{u^{k}}{k!}Q_{k}^{(\xi_1)}(x)\sum_{l=0}^{\infty}\frac{u^l}{l!}Q_{l}^{(\xi_2)}(y)\\
&=\sum_{k=0}^{\infty}\sum_{l=0}^{\infty}\frac{u^{k+l}}{k! l!}Q_{k}^{(\xi_1)}(x)Q_{l}^{(\xi_2)}(y).
\end{align*}
Substituting $n=k+l$ yields (\ref{independent}). Notice that by choosing $\xi_2\equiv 0$ then $Q_{k}^{(0)}(x)=x^k$ the series expansion (\ref{series-exp}) can be obtained from (\ref{independent}).
\end{proof}
\begin{corollary}
Let $\xi_1, \xi_2$ and $\xi$ be as in Proposition \ref{propo-i.i.d}. Then
\begin{equation}\label{indepen-represen}
Q_{n}^{(\xi_1)}(x)=\int_{-\infty}^{\infty}Q_{n}^{(\xi)}(x+y)\Pro(\xi_2\in dy).
\end{equation}
\end{corollary}
\begin{remark}
If $\xi_1, \xi_2$ are i.i.d symmetric random variables then $\xi=\xi_1+\xi_2$ is also symmetric. Therefore, from \eqref{independent-2} and Corollary \ref{symmetric} , we have
\[0=Q_{2n+1}^{(\xi)}(0)=\sum_{k=0}^{2n+1}\binom{2n+1}{k}(-1)^kQ_{k}^{(\xi_1)}(x)Q_{2n+1-k}^{(\xi_1)}(-x).\] 
\end{remark}
\subsubsection{Moment representation} 
 
Let $\eta$ be a complex or real random variable with finite moments of all orders. Clearly, the polynomial
\[\hat{Q}_n(x):=\E(x+\eta)^n\]
satisfies Appell differential equation, see (\ref{recursive}): 
$$\displaystyle\frac{d}{dx}\hat{Q}_{n}(x)=n\hat{Q}_{n-1}(x).$$
A natural question is whether there exists a random variable $\theta$ such that
\[Q_{n}^{(\theta)}(x)=\hat{Q}_n(x),\]
i.e., 
\begin{equation}\label{moment-repre-1}
Q_{n}^{(\theta)}(x)=\E(x+\eta)^n.
\end{equation}
This equality is then called the {\it moment representation} of the Appell polynomial. 
\begin{proposition}
The only pairs of real random variables satisfying (\ref{moment-repre-1}) are the deterministic ones, i.e., $\theta=-\eta=c$ for some constant $c\in \R$.
\end{proposition}
\begin{proof}
The Appell polynomial $Q_{n}^{(\theta)}$ associated with $\theta$ satisfies $$\E(Q_{n}^{(\theta)}(\theta))=0.$$ 
Hence, from (\ref{moment-repre-1})
\begin{equation}\label{product-exp}
\E_{\theta\times\eta}((\theta+\eta)^n)=0,
\end{equation}
where $\E_{\theta\times\eta}$ denotes the product measure induced by $\theta$ and $\eta$. Consequently, taking $n=2$, it is seen that (\ref{product-exp}) holds if and only if \begin{equation}\label{product-pro}
\theta+\eta=0,\quad \Pro_{\theta\times\eta}-a.s. 
\end{equation}
But $\theta$ and $\eta$ are independent under $\Pro_{\theta\times\eta}$ and, hence, it follows from (\ref{product-pro}) that there exists a constant $c$ such that $\theta=-\eta=c$.
\end{proof}
Although there does not exists non-constant real random variables satisfying (\ref{moment-repre-1}) it is possible in some interesting cases to find for a given real $\theta$ a complex random variable $\eta$ such that (\ref{moment-repre-1}) holds. To develop this, consider the following, 

\begin{proposition}\label{moment-propo}
Let $\theta$ be a random variable with the moment generating function
\[M_{\theta}(u):= \E(e^{u\theta}),\]
which is assumed to be well defined and positive for all $u\in\R$. Then if 
\[\psi(u):=\frac{1}{M_{\theta}(u)}\]
is the characteristic function of a real random variable $\zeta$ having finite moments of all orders then
\begin{equation}\label{moment-repre-2}
Q_{n}^{(\theta)}(x)=\E(x+i\zeta)^n,
\end{equation}
where $i$ is the imaginary unit. In particular, $$Q_{2n}^{(\theta)}(0)=(-1)^n\E(\zeta^{2n}).$$
\end{proposition}

\begin{proof}
We have
\begin{equation}\label{eq-1}
\frac{e^{ux}}{M_{\theta}(u)}=\frac{e^{ux}}{\E(e^{u\theta})}=\sum_{n=0}^{\infty}\frac{u^n}{n!}Q_{n}^{(\theta)}(x),
\end{equation}
and 
\begin{align}\label{eq-2}
\nonumber\frac{e^{ux}}{M_{\theta}(u)}=e^{ux}\E(e^{iu\zeta})&=\E(e^{u(x+i\zeta)})\\
&=\sum_{n=0}^{\infty}\E(x+i\zeta)^n.
\end{align}
The claim follows now from (\ref{eq-1}) and (\ref{eq-2}).
\end{proof}
\begin{remark}
Notice that $\zeta$ in Proposition  \ref{moment-propo}  is symmetric (since its characteristic function is real valued). Hence, we have the relationship
\[\frac{1}{\E(e^{u\theta})}=\E(e^{iu\zeta})=\E(\cos(u\zeta)),\]
and, consequently, also $\theta$ must be symmetric. In fact, for symmetric $\theta$ having positive moment generating function for all $u\in\R$ it holds
\item{(i)} $\displaystyle\frac{d}{du}M_{\theta}(u)|_{u=0}=0$,
\item{(ii)} $u\mapsto M_{\theta}(u)$ is convex, increasing for $u\geq 0$ and decreasing for $u\leq 0$. 
\item{(iii)} $M_{\theta}(u)\geq 1$ for all $u$ and $\lim_{u\rightarrow\infty}M_{\theta}(u)=+\infty$.
\end{remark}
In the next section we give several examples of random variables such that their Appell polynomials have the moment representations.

\subsubsection{ Appell polynomials with unique positive root} \label{root}
In this section we will give  conditions for a random variable which guarantee that its Appell polynomials have a unique positive root. This property plays an important role in studying optimal stopping problems of power reward function for L\'evy processes, for more details we refer to \cite{NS, KS, Sa}.

Let us first recall Descartes' rule of signs (see, e.g., \cite{wang}): {\it Let $p_n(x)=\sum_{k=0}^{n}\lambda_{k} x^{\sigma_{k}}$ be a polynomial with non-zero real coefficients $\lambda_{k}$ and powers $\sigma_k$ which are integers satisfying $0\leq\sigma_0<\sigma_1<\dots<\sigma_n.$ Then the number of positive zeros of $p_n(x)$ (counted with multiplicities) is either equal to the number of variations in sign in the sequence $\{\lambda_k\}$ or less than that by an even integer number.}

\begin{proposition}
Let $\xi$ be a non-negative random variable having the moments up to $n$. If the following conditions hold for all $a>0$
\begin{description}
\item{(i)}  
$\quad \Pro(\xi<a)>0$,
\item{(ii)}
$\quad \E(Q_{k}^{(\xi)}(\xi)1_{\{\xi\geq a \}})>0, \quad k=1,2,\dots,n$,
\end{description}
then the Appell polynomials $Q_{k}^{(\xi)}$, $k=1,2,\dots,n$ have a unique positive root.
\end{proposition}
\begin{proof}
We show that from condition (ii) it follows that $Q_{k}^{(\xi)}(0)\leq 0$, $k=1,2\dots, n$. Indeed, assume that $Q_{k}^{(\xi)}(0)>0$, by the continuity of $Q_{k}^{(\xi)}$ there exists $\varepsilon>0$ and $\delta(\varepsilon)>0$ such that $Q_{k}^{(\xi)}(z)>\varepsilon$ for all $z\in (-\delta(\varepsilon), \delta(\varepsilon))$. Choosing $a\in(0, \delta(\varepsilon)) $ we have by (i)
\begin{equation}\label{uni-2}
\E(Q_{k}^{(\xi)}(\xi)1_{\{\xi< a \}})>\varepsilon \  \Pro(\xi<a)>0.
\end{equation}
Since $\E(Q_{k}^{(\xi)}(\xi))=0$ it follows  from (ii)
\[\E(Q_{k}^{(\xi)}(\xi)1_{\{\xi< a\}})=-\E(Q_{k}^{(\xi)}(\xi)1_{\{\xi\geq a \}})<0,\]
which contradicts \eqref{uni-2}. Hence $Q_{k}^{(\xi)}(0)\leq 0$ for all $k=1, 2\dots,n$.
Furthermore, we have $Q_{0}^{(\xi)}(0)=1$ and
$\displaystyle Q_{k}^{(\xi)}(x)=\sum_{i=0}^{k}\binom{k}{i}Q_{i}^{(\xi)}(0)x^{k-i}$. Consequently, the coefficients of $Q_{k}^{(\xi)}$ have only one change of sign. By Descartes' rule of signs we conclude that $Q_{k}^{(\xi)}, k=1,2\dots,n$ have a unique positive zero.
\end{proof}

\section{Examples of Appell polynomials}
In this section we consider some well known Appell polynomials, e.g., Bernoulli, Euler, Hermite and Laguerre polynomials and connect them with random variables. We use the characterization via random variables to review some properties of these polynomials. 

\subsection{Bernoulli polynomials and uniform distribution}
Let $\xi$ $\sim$ {\rm U}$(0,1).$  Then 
\[\E(e^{u\xi})=\frac{e^u-1}{u}.\]
The Bernoulli polynomials are introduced via (see \citep[p 804]{AS})
\[\frac{ue^{ux}}{e^u-1}=\sum_{n=0}^{\infty}\frac{u^n}{n!}B_n(x).\]
Hence, the Appell polynomials associated with the uniformly distributed random variable coincide with the Bernoulli polynomials. The numbers $\tilde{B}_{k}:=B_{k}(0), k=0,1,\dots $ are called the Bernoulli numbers. We have $$\E(\xi^k)=\int_{0}^{1}y^kdy=1/(k+1), \quad k=1,2,\dots.$$ 
From (\ref{app-def2-3}) we obtain a recursive formula for the Bernoulli numbers: 
\begin{equation}\label{ber-numb}
\tilde{B}_n=-\sum_{k=0}^{n-1}\binom{n}{k}\tilde{B}_k\frac{1}{n-k+1},
\end{equation}
which gives $$\tilde{B}_0=1, \tilde{B}_1=-\frac{1}{2}, \tilde{B}_2=\frac{1}{6}, B_3=0, B_4=-1/30\dots.$$
Notice also that
\begin{equation}\label{even-func}
\frac{u}{e^u -1}+\frac{u}{2}-1 =\sum_{n=2}^{\infty}\frac{u^n}{n!}B_n(0).
\end{equation}
Since the left hand side of (\ref{even-func}) defines an even function it follows that $\tilde{B}_{2k+1}=0$ for $k=1,2,\dots$.
Furthermore, from (\ref{appell-rep}) we have
\[B_{n}(x)=\sum_{m=0}^{n}\binom{n}{m}\tilde{B}_{n-m}x^m.\]
We now exploit the properties of the Appell polynomials under the probabilistic approach to provide new proofs of many properties of the  Bernoulli polynomials; some of these can be found in \citep[p 804]{AS}). First we prove the relationship between the cumulants $\kappa_n$ of $\xi$ and the Bernoulli numbers $\tilde{B}_n$. Clearly, the first cumulant $$\kappa_1=\E(\xi)=1/2=-\tilde{B}_1.$$

\noindent
{\it Property {\rm (i)}.} {\it Cumulant relationship:}
\begin{equation}\label{cum-ber-numb}
\kappa_{n+1}=\frac{\tilde{B}_{n+1}}{n+1}, \quad n=1,2\dots.
\end{equation}

\begin{proof}
From (\ref{ber-numb})
\begin{align*}
\tilde{B}_{n+1}&=-\sum_{j=0}^{n}\binom{n+1}{j}\tilde{B}_j\frac{1}{n-j+2}\\
&=-(n+1)\sum_{j=0}^{n}\binom{n}{j}\tilde{B}_j\frac{1}{n-j+1}\frac{1}{n-j+2}\\
&=-(n+1)\sum_{j=0}^{n}\binom{n}{j}\tilde{B}_j\frac{1}{n-j+1}+(n+1)\sum_{j=0}^{n}\binom{n}{j}\tilde{B}_j\frac{1}{n-j+2}.
\end{align*}
In the last equality, the first sum is 0 by (\ref{ber-numb}),  and using (\ref{cumulant-mu-q}) the second sum equals $\kappa_{n+1}$ and we therefore obtain formula (\ref{cum-ber-numb}).
\end{proof}
\noindent
{\it Property {\rm(ii)}.} {\it First difference property:}
\begin{equation}\label{diff-B}
B_n(x+1)-B_n(x)=nx^{n-1},
\end{equation}
{\it which yields}
\[\sum_{k=1}^{m}k^n=\frac{B_{n+1}(m+1)-B_n(m)}{n+1}.\]
\begin{proof}
Using the mean value property of the Appell polynomials we get
\begin{align*}
x^{n-1}=\E(Q_{n-1}^{(\xi)}(x+\xi))&=\sum_{m=0}^{n-1}\binom{n-1}{m}\tilde{ B}_m\E(x+\xi)^{n-1-m}\\
&=\sum_{m=0}^{n-1}\binom{n-1}{m}\tilde{B}_m\frac{1}{n-m}\Big[(x+1)^{n-m}-x^{n-m}\Big]\\
&=\frac{1}{n}\sum_{m=0}^{n-1}\binom{n}{m}\tilde{B}_m\Big[(x+1)^{n-m}-x^{n-m}\Big]\\
&=\frac{1}{n}[B_n(x+1)-B_n(x)],
\end{align*}
implying (\ref{diff-B}).
\end{proof}
\noindent
{\it Property {\rm (iii)}.} {\it Symmetry:}
\begin{equation}\label{symmetry}
B_n(1-x)=(-1)^{n}B_n(x).
\end{equation}
\begin{proof}
Set $\eta:=1-\xi$. Then also $\eta\sim {\rm U}(0,1)$.  Therefore, the Appell polynomials associated with $\eta$ are also the Bernoulli polynomials $B_n(x)$.\\
On the other hand,
\[\frac{e^{ux}}{\E(e^{u\eta})}=\frac{e^{u(x-1)}}{\E(e^{u(-\xi)})},\]
and combining this with (\ref{relation}) gives
\[Q_{n}^{(\eta)}(x)=Q_{n}^{(-\xi)}(x-1)=(-1)^{n}Q_{n}^{(\xi)}(1-x),\]
proving (\ref{symmetry}).\\
\end{proof}
\noindent
{\it Property {\rm (iv)}.} {\it Second difference property:}
\[(-1)^nB_n(-x)-B_n(x)=nx^{n-1}.\]
\begin{proof}
This can be obtained from (\ref{diff-B}) and (\ref{symmetry}).
\end{proof}
\noindent
{\it Property {\rm(v)}.} {\it Representation of powers:}
\begin{equation}\label{ber-1}
x^n=\frac{1}{n+1}\sum_{k=0}^{n}\binom{n+1}{k}B_k(x).
\end{equation}
\begin{proof}
From (\ref{inverse-formula}) and the fact $\displaystyle\E(\xi^{n-k})=\frac{1}{n-k+1}$ we have
\[x^n=\sum_{k=0}^{n}\binom{n}{k}\frac{1}{n-k+1}B_k(x),\]
and this implies (\ref{ber-1}).\\
\end{proof}
Next we give a new proof of the moment representation for Bernoulli polynomials first presented by Sun \citep{Sun} (see also Srivastava and Vignat \cite{S-V}).\\

\noindent
{\it Property {\rm(vi)}.} {\it Moment representation:}
\begin{equation}\label{ber-moment-representation}
B_n(x)=\E(x-\frac{1}{2}+i\zeta)^n,
\end{equation}
{\it where $\zeta$ has the logistic distribution with the density 
\[f^{(\zeta)}(x)=\frac{\pi}{2}{\rm sech}^2(\pi x), \quad x\in \R.\]}
\begin{proof}
 We make a symmetric random variable from the uniform random variable $\xi$ by putting $\eta:=\xi-1/2$. Then $\eta$ has the uniform distribution on $[-1/2, 1/2]$. We have
\[\E(e^{u\eta})=e^{-u/2}\E(e^{u\xi})=\frac{e^{u/2}-e^{-u/2}}{u},\] 
and, hence,
\[Q_{n}^{(\eta)}(x)=B_n(x+1/2).\]
Consider the cosine Fourier transform
\[\E(\cos(u\zeta))=\frac{1}{\E(e^{u\eta})}=\frac{u}{e^{u/2}-e^{-u/2}}=\frac{u/2}{\sinh(u/2)},\]
The inverse of this transform is given by (see \cite[(2)p 30]{erdelyi54})
\[f^{(\zeta)}(x)=\frac{\pi}{2}\frac{1}{\cosh^2(\pi x)}=\frac{\pi}{2}{\rm sech}^2(\pi x).\]
Therefore, we have the moment representation
\[Q_{n}^{(\eta)}(x)=\E(x+i\zeta)^n,\]
and, hence,
\[B_n(x)=Q_{n}^{(\eta)}(x-\frac{1}{2})=\E(x-\frac{1}{2}+i\zeta)^n.\]
\end{proof}

The following result follows from the formula of the Appell polynomials associated to the sum of two independent variables.\\

\noindent
{\it Property {\rm(vii)}.} {\it For $n=0,1,2,\dots$}
\begin{equation}\label{prop-6}
B_n(x)=2^{n-1}\Big(B_n(\frac{x}{2})+B_n(\frac{x+1}{2})\Big).
\end{equation}
\begin{proof}
Let $\eta\sim U(0,1)$ and $\xi\sim Ber(1/2)$, i.e., $\Pro(\xi=0)=\Pro(\xi=1)=1/2$. Assume that $\eta$ and $\xi$ are independent. Then (see \citep[Proposition 1]{ta})
\begin{equation*}
Q_{n}^{(\eta+\xi)}(x)=2^n B_n(x/2).
\end{equation*}
From which and (\ref{indepen-represen}) we obtain
\[Q_{n}^{(\eta)}(x)=\E(Q_{n}^{(\eta+\xi)}(x+\xi))=2^{n-1}\Big(B_n(\frac{x}{2})+B_n(\frac{x+1}{2})\Big).\]
But $B_n(x)\equiv Q_{n}^{(\eta)}(x)$. So the proof is complete.
\end{proof}
\noindent
{\it Property {\rm (viii)}.} {\it For $n=0,1,2\dots$}
\begin{equation}\label{value-one-two}
B_n(1/2)=-(1-2^{1-n})B_n(0)
\end{equation}

\begin{proof}
Formula (\ref{value-one-two}) is equivalent with
\[B_n(0)=2^{n-1}\Big(B_n(1/2)+B_n(0)\Big),\]
which is obtained directly by putting $x=0$ in (\ref{prop-6}).
\end{proof}
\noindent
We next give a simple proof for the sign of the Bernoulli numbers (see \cite{carlitz-scoville})\\

\noindent
{\it Property {\rm (ix)}.} {\it The sign of the Bernoulli numbers:}
\begin{equation}\label{sign-b-number}
(-1)^{n-1}B_{2n}(0)>0, \quad\text{for all}\quad n\geq 1.
\end{equation}
\begin{proof}
From the moment representation (\ref{ber-moment-representation}) we have
\[B_{2n}(1/2)=(-1)^n\E(\zeta^{2n}).\]
Combining this fact with (\ref{value-one-two}) yields
\[(-1)^{n-1}B_{2n}(0)=\frac{\E(\zeta^{2n})}{(1-2^{1-2n})}>0,\]
where $\zeta$ is as in Property (vi), from which readily follows (\ref{sign-b-number}).
\end{proof}

\subsection{Euler polynomials and Bernoulli distribution}
Let $\xi\sim Ber(1/2)$, i.e.,
\begin{equation}\label{ber-dist}
 \Pro(\xi=0)=\Pro(\xi=1)=\frac{1}{2}.
\end{equation} 
Then
\[\E(e^{u\xi})=\frac{e^u+1}{2},\]
and, hence, 
\begin{equation}\label{eur-def}
\frac{e^{ux}}{\E(e^{u\xi})}=\frac{2e^{ux}}{e^u+1}=\sum_{n=0}^{\infty}\frac{u^n}{n!}E_n(x),
\end{equation}
where $E_n(x)$ are the so called the Euler polynomials (see \citep[p 804]{AS}). Consequently, the Appell polynomials associated with $\xi$ coincide with the Euler polynomials and we have from (\ref{series-exp})
\[E_n (x)=\sum_{m=0}^{n}\binom{n}{m}\tilde{E}_{n-m}x^m,\]
where the numbers $\tilde{E}_{k}=E_{k}(0), k=0,1\dots n $  can be generated from the recurrence equation (\ref{app-def2-3}) which - since $\E(\xi^k)=1/2$ - takes the form
\begin{equation}\label{eu-n-zero}
\tilde{E}_n=-\frac{1}{2}\sum_{k=0}^{n-1}\binom{n}{k}\tilde{E}_k.
\end{equation}
Calculating from (\ref{eu-n-zero}) we obtain 
$$\tilde{E}_0=1, \tilde{E}_1=-\frac{1}{2}, \tilde{E}_2=0, \tilde{E}_3=\frac{1}{4}, \tilde{E}_4=0, \tilde{E}_5=-\frac{1}{2},\dots.$$
Notice from (\ref{eur-def}) that
\[\frac{2}{e^u+1}+\frac{u}{2}-1=\sum_{n=2}^{\infty}\frac{u^n}{n!}\tilde{E}_n.\]
Consequently, since the function on the LHS is odd it follows that  $\tilde{E}_{2k}=0$ for $k\geq1$. 
 
Similarly as for the Bernoulli polynomials, we also discuss the following properties (see also \citep[p 804]{AS}) of the Euler polynomials via the probabilistic approach:\\

\noindent
{\it Property} (i). {\it Difference property:}
\begin{equation}\label{diff-E}
E_n(x+1)+E_n(x)=2x^n,
\end{equation} 
{\it which yields}
\begin{equation}\label{eu-sum-power}
\sum_{k=0}^{m}(-1)^k k^n=\frac{(-1)^m E_{n}(m+1)+E_n(0)}{2}.
\end{equation}
\begin{proof}
Using again the mean value property, we get
\[x^n=\E(Q_{n}^{(\xi)}(x+\xi))=\sum_{m=0}^{n}\binom{n}{m}\tilde{E}_{m}\frac{1}{2}\Big[(x+1)^{n-m}+x^{n-m}\Big],\]
which yields (\ref{diff-E}). The proof of (\ref{eu-sum-power}) is obtained from (\ref{diff-E}) by observing
\[\sum_{k=0}^{m}(-1)^k k^n=\sum_{k=0}^{m}\frac{(-1)^kE_n(k)-(-1)^{k+1}E_{n}(k+1)}{2}.\]
\end{proof}

\noindent
{\it Property} (ii). {\it Symmetry:}
\begin{equation*}
E_n(1-x)=(-1)^{n}E_n(x).
\end{equation*}

\noindent
{\it Property} (iii). {\it Representation of powers:}
\begin{equation*}
x^n=E_n(x)+\frac{1}{2}\sum_{k=0}^{n-1}\binom{n}{k}E_k(x).
\end{equation*}

\noindent
{\it Property} (iv). {\it Euler numbers:}
$$\hat{E}_n:=2^nE_n(\frac{1}{2}),\quad n=0,1,2,\dots.$$
{\it the so-called Euler numbers, are integers}.
\begin{proof}
From (iii), putting $x=1/2$ we have the recurrence for the Euler numbers
\[\hat{E}_n=1-\sum_{k=0}^{n-1}\binom{n}{k}2^{n-k-1}\hat{E}_k.\]
Using induction here it easily seen that the Euler numbers are integers.\\
\end{proof}

\noindent
{\it Property} (v). {\it Relationship between the Bernoulli and Euler polynomials:}.
\begin{equation}\label{ber-eu-repre}
E_{n}(x)=\frac{2^{n+1}}{n+1}\Big[B_{n+1}\big(\frac{x+1}{2}\big)-B_{n+1}\big(\frac{x}{2}\big)\Big].
\end{equation}
\begin{proof}
Let $\eta$ and $\xi$ be two independent random variables with $\eta\sim U(0,1)$ and $\xi\sim Ber(1/2)$. Then as in property (vi) for Bernoulli polynomials,
\begin{equation}\label{ber-eur-rela}
Q_{n}^{(\eta+\xi)}(x)=2^nB_n(x/2).
\end{equation}
Using (\ref{ber-eur-rela}) and (\ref{indepen-represen}) we obtain
\[E_n(x)=Q_{n}^{(\xi)}(x)=\int_{0}^{1}Q_{n}^{(\eta+\xi)}(x+y)dy=\int_{0}^{1}2^nB_n\big(\frac{x+y}{2}\big)dy.\]
Setting $z=(x+y)/2$  we now have
\begin{equation}\label{ber-eur-rel-2}
E_n(x)=2^{n+1}\int_{\frac{x}{2}}^{\frac{x+1}{2}}B_n(z)dz=2^{n+1}\Big[\frac{B_{n+1}\big(\frac{x+1}{2}\big)-B_{n+1}\big(\frac{x}{2}\big)}{n+1}\Big].
\end{equation}
Notice that the last equality of (\ref{ber-eur-rel-2}) is obtained from  recursive differential equation (\ref{recursive}). 
\end{proof}

\noindent
{\it Property} (vi). {\it Moment representation:
\[E_n(x)=\E(x+i\zeta)^n,\]
where $\zeta$ has the hyperbolic secant distribution with the density
\[f^{(\zeta)}(y)={\rm sech}(\pi y).\]
In particular, the Euler numbers satisfy
\[\hat{E}_{2n}=(-1)^n 2^{2n}\E(\zeta^{2n})\quad \text{and}\quad \hat{E}_{2n+1}=0, \quad n\geq 1,\]
from which it follows the sign property of the Euler numbers (see also \cite{carlitz-scoville})
 $$(-1)^n\hat{E}_{2n}> 0.$$
}
\begin{proof}
Let $\theta:=\xi-1/2$. Then $\theta$ is symmetric and we get
\[\E(e^{u\theta})=e^{-u/2}\E(e^{u\xi})=\frac{e^{u/2}+e^{-u/2}}{2}.\]
Hence, 
\[Q_{n}^{(\theta)}(x)=E_n(x+\frac{1}{2}).\]
We have the cosine transform \[\E(\cos(u\zeta))=\frac{1}{\E(e^{u\eta})}=\frac{2}{e^{u/2}+e^{-u/2}}=\frac{1}{\cosh(u/2)},\]
and, hence, inverting this transform gives (see \cite[(1)p 30]{erdelyi54})
 \[f^{(\zeta)}(y)={\rm sech}(\pi y),\]
which results to the moment representation for $Q_{n}^{(\theta)}$
\[Q_{n}^{(\theta)}(x)=\E(x+i\zeta)^n.\]
 Consequently (see also \cite{Sun})
 \[E_n(x)=\E(x-\frac{1}{2}+i\zeta)^n.\]
 \end{proof}
 \begin{remark}\label{talacko}
Talacko \citep{Talacko} considers a class of symmetric distributions which is called the family of Perks' distributions with the density function
\[f(x)=\frac{c}{e^x+k+e^{-x}},\]
 where $c$ is the normalizing constant and $k>-2$, $k\in \R$ . For $k=0$ we have the so-called hyperbolic secant distribution and for $k=2$ the logistic distribution. In \citep{Talacko} the following characterizations are also proved:\\
 \noindent
(i) Let $\xi_2$ be a random variable having the logistic distribution, then
 $$\xi_2\stackrel{(d)}=\sum_{j=1}^{\infty}\frac{L_j}{2j\pi},$$
 where $L_i$ are i.i.d of sequence of Laplace distributed random variables with the density $\displaystyle g(x)=\frac{1}{2}e^{-|x|}$.\\
\noindent 
 (ii) Let $\xi_0$ have the hyperbolic secant distribution, then  
   $$\xi_0\stackrel{(d)}=\sum_{j=1}^{\infty}\frac{L_j}{(2j-1)\pi}.$$
Sun \citep{Sun} uses these characterizations to give the moment representations for Bernoulli polynomials and Euler polynomials. Such representations are also used in Srivastava and Vignat  \citep{S-V}.
 \end{remark}
 
\subsection{Hermite polynomials and normal distribution}
Let $\xi\sim N(0,1)$. Then
\[\E(e^{u\xi})=e^{u^2/2},\]
and, hence,
\[\frac{e^{ux}}{\E(e^{u\xi})}=e^{ux-u^2/2}.\]
So the Appell polynomials $Q_{n}^{(\xi)}$ associated with $\xi$ are the Hermite polynomials $He_n$ (see also \citep{Sa}): 
\[Q_{n}^{(\xi)}(x)=He_n(x)= n!\sum_{k=0}^{[n/2]}\frac{(-1)^k x^{n-2k}}{(n-2k)! k! 2^k}.\]
\begin{remark}
Let $\hat{\xi}\sim N(\mu,\sigma^2)$ then $\hat{\xi}\stackrel{(d)}=\mu+\sigma\xi$. Denote by $H_n$  the Appell polynomials associated with $\hat{\xi}$, from (\ref{appell}) we have 
\[H_n(x)=\sigma^n He_n\Big(\frac{x-\mu}{\sigma}\Big).\]
\end{remark}
\noindent
We discuss now some basis properties of $He_n$.\\

\noindent
{\it Property} (i). {\it Symmetry:}
\[He_n(x)=(-1)^n He_{n}(-x).\]
\begin{proof}
Since the random variable $\eta:=-\xi$ has also the standard normal distribution, the property follows from (\ref{relation}).
\end{proof}

\noindent
{\it Property} (ii). {\it Moment representation:}
\[He_{n}(x)=\E(x+i\xi)^n,\]
{\it In particular,}
\[He_{2n}(0)=(-1)^n\E(\xi^{2n})=(-1)^n\frac{(2n)!}{2^n n!}.\]
\begin{proof}
It is seen that $\xi$ is symmetric and the function $$\psi(u)=(\E(e^{u\xi}))^{-1}=e^{u^2/2}$$ is well defined for all $u$. 
Consider the cosine transform
\[\E(\cos(u\zeta))=2\int_{0}^{\infty}f^{(\zeta)}(y)\cos(uy)dy=e^{-u^2/2}.\]
From \cite[(11) p 15]{erdelyi54}
\[\int_{0}^{\infty}\frac{e^{-y^2/2}}{\sqrt{2\pi}}\cos(uy)dy=\frac{1}{2}e^{u^2/2},\]
and, hence, $\zeta\stackrel{(d)}=\xi \sim N(0,1)$  and $\displaystyle\E(\zeta^{2n})=\frac{(2n)!}{2^n n!}$.
\end{proof}
\begin{remark}
Withers \cite{withers} derives the moment representation for multiple Hermite polynomials by using Taylor expansion of $\E(e^{u(x+iZ)})$, where $Z$ is a multidimensional normally distributed random variable. 
\end{remark}

\noindent
{\it Property} (iii). {\it Representation of powers:}
\[x^{2n}=\sum_{i=1}^{n}\binom{2n}{2k}\frac{(2(n-k))!}{2^{n-k}(n-k)!}He_{2k}(x).\]

Next we consider the orthogonality of Hermite polynomials. As proved in \cite{Car2} (see also \citep{So}) the Hermite polynomials are the only Appell polynomials which are orthogonal.  Here we give a new proof of this fact via the moment generating function. Recall that by Favard's theorem (see \cite[p 21]{Chihara}, \cite{Sch}) the family of polynomials $\{P_n, P_0\equiv 1, n=1,2,\dots\}$ is orthogonal if and only if there exist sequences $\{a_n\}_{n=1}^{\infty}$ and  $\{b_n\}_{n=1}^{\infty}$ such that 
\[P_{n+1}(x)=(x-a_n)P_{n}(x)-b_n P_{n-1}(x), \quad n=1,2,\dots.\]

\noindent
{\it Property} (iv). {\it Let $\xi$ be a random variable having some exponential moments. The Appell polynomials $\{Q_{n}^{(\xi)}\}_{n\geq 0}$ associated $\xi$  are orthogonal if and only if $\xi$ is normally distributed, i.e., $Q_{n}^{(\xi)}$ are Hermite polynomials $He_n$, $n=0,1\dots$.}

\begin{proof}
Since it is assumed that $\xi$ has some exponential moments, the cumulants $\kappa_n, n=1,2,\dots$ are obtained via formula
\[\E(e^{u\xi})={\rm exp}\big\{\sum_{n=1}^{\infty}\kappa_n\frac{u^n}{k!}\big\}.\] 
From (\ref{appell}) we have
\begin{equation*}{\label{ortho}}
{\rm exp}\big\{ux-\sum_{n=1}^{\infty}\kappa_n\frac{u^n}{n!}\big\}=\sum_{n=0}^{\infty}\frac{u^n}{n!}Q_{n}^{(\xi)}(x).
\end{equation*}
Taking derivative in $u$ yields
\[\Big(x-\sum_{n=1}^{\infty}\frac{u^{n-1}}{n!}n\kappa_n\Big)\exp\Big\{ux-\sum_{n=1}^{\infty}\frac{u^n}{n!}\kappa_n\Big\}=\sum_{n=1}^{\infty}\frac{u^{n-1}}{n!}nQ_{n}^{(\xi)}(x).\]
It follows that
\[\Big(x-\sum_{n=0}^{\infty}\frac{u^{n}}{n!}\kappa_{n+1}\Big)\sum_{n=0}^{\infty}\frac{u^n}{n!}Q_{n}^{(\xi)}(x)=\sum_{n=1}^{\infty}\frac{u^{n-1}}{n!}nQ_{n}^{(\xi)}(x),\]
which is equivalent with
\[x\sum_{n=0}^{\infty}\frac{u^n}{n!}Q_{n}^{(\xi)}(x)-\sum_{i=0}^{\infty}\sum_{j=0}^{\infty}\kappa_{i+1}Q_{j}^{(\xi)}(x)\frac{u^{i+j}}{i!j!}=\sum_{n=0}^{\infty}\frac{u^n}{n!}Q_{n+1}^{(\xi)}(x).\]
Substituting  $n=i+j$ gives
\[x\sum_{n=0}^{\infty}\frac{u^n}{n!}Q_{n}^{(\xi)}(x)-\sum_{n=0}^{\infty}\frac{u^{n}}{n!}\sum_{j=0}^{n}\binom{n}{j}\kappa_{n+1-j}Q_{j}^{(\xi)}(x)=\sum_{n=0}^{\infty}\frac{u^n}{n!}Q_{n+1}^{(\xi)}(x).\]
Consequently,
\begin{equation*}\label{represent-app-0}
xQ_{n}^{(\xi)}(x)-\sum_{j=0}^{n}\binom{n}{j}\kappa_{n+1-j}Q_{j}^{(\xi)}(x)=Q_{n+1}^{(\xi)}(x),
\end{equation*}
and, hence, we obtain the representation 
\begin{equation}{\label{represent-app}}
Q_{n+1}^{(\xi)}(x)=(x-\kappa_1)Q_{n}^{(\xi)}(x)-n\kappa_2 Q_{n-1}^{(\xi)}(x)-\sum_{j=0}^{n-2}\binom{n}{j}\kappa_{n+1-j}Q_{j}^{(\xi)}(x).
\end{equation}
It follows from Favard's theorem that $Q_{n}^{(\xi)}$ is orthogonal if and only if $\kappa_k=0$ for all $k\geq 2,$  in other words, if and only if $\xi$ is normally distributed. So we have that $Q_{n}^{(\xi)}$, $n=0,1,2\dots$ are the Hermite polynomial. Furthermore, in case of standard normal, i.e., the cumulants $\kappa_1=0, \kappa_2=1$, we have the recurrence equation 
\[He_{n+1}(x)=xHe_n(x)-nHe_{n-1}(x).\]
\end{proof} 
\begin{remark}
The polynomials $Q_{n}^{(\xi)}$ can be represented in terms of the Bell polynomials $Be_n$ as follows (see Madan and Yor \cite{madanyor})
\[Q_{n}^{(\xi)}(x)=Be_n(x-\kappa_1, -\kappa_2,\dots, \kappa_n),\]
and then by using formula (8) in \citep{SU} we also obtain formula \eqref{represent-app}.
\end{remark}

\subsection{Laguerre polynomials and gamma distribution}

Let $\xi\sim\Gamma(\beta, \alpha), \alpha, \beta>0$, i.e., $\xi$ is a gamma distributed random variable with parameters $\alpha$ and $\beta$. Then 
\[\E(e^{u\xi})=\Big(\frac{\alpha}{\alpha-u}\Big)^{\beta},\quad |u|<\alpha.\]
Take now $\alpha=1$. We have
$$\frac{e^{u x}}{\E(e^{u\xi})}=(1-u)^{\beta}e^{ux}=\sum_{n\geq 0}\frac{u^n}{n!}\sum_{m=0}^{n}\frac{n!}{m!}\binom{\beta}{n-m}(-1)^{n-m}x^m,$$
where 
\[\binom{\beta}{n-m}=\frac{\beta(\beta-1)\dots(\beta-(n-m)+1)}{(n-m)!}.\]
Consequently, the Appell polynomials $Q_{n}^{(\xi)}$ are:
\begin{equation}\label{appell-gamma}
Q_{n}^{(\xi)}(x)=\sum_{m=0}^{n}(-1)^{n-m}\frac{n!}{m!}\binom{\beta}{n-m}x^m.
\end{equation}
In case  $\beta$ is an integer we obtain
\[Q_{n}^{(\xi)}(x)=\sum_{m=\max(0, n-\beta)}^{n}(-1)^{n-m}\frac{n!}{m!}\binom{\beta}{n-m}x^m.\]
Recall that the Laguerre polynomials $L_{n}^{(\gamma-n)}, \gamma>-1$ are defined by (see \citep[p 189]{erdlyi53})
 \begin{equation}\label{laguerre}
 (1+u)^{\gamma}e^{-ux}=\sum_{n=0}^{\infty}L_{n}^{(\gamma-n)}(x)u^n.
 \end{equation}
Hence,    $Q_{n}^{(\xi)}\equiv(-1)^n n!L_{n}^{(\beta-n)}$, in other words, the Laguerre polynomials $L_{n}^{(\beta-n)}$ are explicitly given by 
\begin{equation*}
L_{n}^{(\beta-n)}(x)=\sum_{m=0}^{n}(-1)^{2n-m}\binom{\beta}{n-m}\frac{x^m}{m!}=\sum_{m=0}^{n}(-1)^{m}\binom{\beta}{n-m}\frac{x^m}{m!}.
\end{equation*}
 As explained in subsection \ref{root} the property that an Appell polynomial has a unique positive root is crucial in the theory of optimal stopping of L\'evy processes. In the next proposition it is seen that if the parameter $\beta\leq 1$ then the gamma distribution has this property.
\begin{proposition}\label{app-gamma}
If $0<\beta\leq 1$ then  $Q_{n}^{(\xi)}(x)$ has a unique positive root $x^*$, $Q_{n}^{(\xi)}(x)$  is negative on $(0, x^*)$ and positive and increasing on $(x^*, \infty)$.
\end{proposition}
\begin{proof}
We have
\[Q_{n}^{(\xi)}(x)=\sum_{m=0}^{n}(-1)^{n-m}\frac{n!}{m!}\binom{\beta}{n-m}x^m=\sum_{m=0}^{n}q_m x^m,\]
where 
\[q_m:=(-1)^{n-m}\frac{n!}{m!}\binom{\beta}{n-m}=(-1)^{n-m}\frac{n!\beta(\beta-1)\dots(\beta-(n-m)+1)}{m!(n-m)!}\]
for all $0\leq m\leq n-1$ and $q_n=1.$  Since $0<\beta\leq 1$ then
\[q_m=(-1)^{2(n-m)-1}\frac{\beta(1-\beta)\dots((n-m)+1-\beta)n!}{(n-m)!m!}.\]
Obviously,
\[q_m\leq 0\quad\text{for all}\quad 0\leq m\leq n-1.\]
By the Descartes' rule of signs the proof is complete.
\end{proof}
\begin{remark}
For arbitrary $\alpha >0$ it holds
\[Q_{n}^{(\xi)}(x)\equiv\frac{(-1)^n n!}{\alpha^n}L_{n}^{(\beta-n)}(x).\]
In particular, for $\beta=1$, i.e., $\xi\sim Exp(\alpha)$ we have (see also \cite{Sa})
 \[Q_{n}^{(\xi)}(x)=(x-\frac{n}{\alpha})x^{n-1}, \quad n=1,2,\dots.\]
\end{remark}
We conclude by presenting the following representation of powers via Laguerre polynomials 
 \[x^n=\sum_{m=0}^{n}(-1)^m\frac{\Gamma(\beta+n-m)}{\Gamma(\beta)}m!L_{k}^{(\beta-n)}(x).\]
To prove this, use (\ref{inverse-formula}) and the fact that 
$$\E(\xi^k)=\frac{\Gamma(\beta+k)}{\Gamma(\beta)}.$$
\noindent
{\bf Acknowledgements.} The author would like to thank Professor Paavo Salminen for helpful discussions and valuable comments which improved this paper. The financial support from the Finnish Doctoral Programme in Stochastics and Statistics is gratefully acknowledged.

\end{document}